\documentclass[12pt,a4paper,twoside,reqno]{amsart}
\usepackage{cite}
\usepackage{amsmath,dsfont}
\usepackage{amssymb}
\usepackage{amsthm}
\usepackage{graphicx}
\usepackage{float}
\usepackage[linkcolor=blue, urlcolor=blue, citecolor=blue, colorlinks, bookmarks]{hyperref}
\usepackage{amsfonts}
\usepackage{mathrsfs}
\usepackage{algorithm}
\allowdisplaybreaks

\allowdisplaybreaks

\usepackage[linkcolor=blue, urlcolor=blue, citecolor=blue, colorlinks, bookmarks]{hyperref}
\usepackage{amsfonts}
\usepackage{mathrsfs}
\usepackage[top=2cm, bottom=2cm, left=2.5cm, right=2.5cm]{geometry} 

\usepackage{amsfonts,amssymb,amscd,amsmath,enumerate,verbatim,calc}
\usepackage{algpseudocode}
\newcommand{\mychi}{\raisebox{0pt}[1ex][1ex]{$\chi$}}
\newcounter{Theorem}

\newtheorem{corollary}[Theorem]{Corollary}
\newtheorem{remark}[Theorem]{Remark}
\newtheorem{definition}[Theorem]{Definition}
\newtheorem{example}[Theorem]{Example}
\newtheorem{theorem}[Theorem]{Theorem}

\numberwithin{Theorem}{section} \numberwithin{equation}{section}

\Large

\begin{document}
	
	\title []{ A class of Zero Divisors and Topological Divisors of Zero in some Banach algebras }
	
	
	\author[]{Anurag Kumar Patel}
	\address{Anurag Kumar Patel \\ Department of Mathematics \\ Banaras Hindu University \\ Varanasi 221005, India}
	\email{anuragrajme@gmail.com}
	
	\author[]{Harish Chandra}
	\address{Harish Chandra \\Department of Mathematics \\ Banaras Hindu University \\ Varanasi 221005, India}
	\email{harishc@bhu.ac.in}
	
	
	
	
	\thanks{The first author is supported by the Council of Scientific and Industrial Research (CSIR) NET-JRF, New Delhi, India, through grant $09/013(0891)/2019-EMR-I$} 
	
	\begin{abstract}
	In this paper, we establish necessary and sufficient conditions that must be met for weighted composition operators to act as zero divisors in $\mathcal{B}(\ell^p).$ We also give a necessary condition and a sufficient condition for a composition operators to act as zero divisors in $\mathcal{B}(L^p(\mu)).$ Subsequently, we characterize TDZ in $C(X)$.
	Afterward, we establish that a multiplication operator $M_h$ in $\mathcal{B}(C(X))$ becomes a TDZ if and only if $h$ is a TDZ in $C(X).$
	Further, motivated by the definition of TDZ, we introduce  notions of polynomially TDZ and strongly TDZ and prove  
	that every element in $C(X)$ and in $L^\infty(\mu)$ is a polynomially TDZ. We then prove that a multiplication operator $M_h$ in $\mathcal{B}(C(X))$ as well as  in $\mathcal{B}(L^p(\mu))$ is a polynomially TDZ. 
	Lastly, we show that each $T\in \mathcal{B}(H)$, where $H$ is a separable Hilbert space, is a strongly TDZ.
	\end{abstract}
	
	\maketitle
	
	\section{Introduction}
	The notion of zero divisor{\cite{Gelfand Silov}} is well known from ring theory. In a ring $\mathfrak{R},$ an element $z$ is termed a left $(right)$ zero divisor if $zx=0~(xz=0)$ for some $0 \neq x\in \mathfrak{R}$ and if it is either a left zero divisor or a right zero divisor, we call it a zero divisor.
	
	 The notion of topological divisors of zero($\textbf{TDZ}$) is an important generalization in a Banach algebra, introduced by Shilov in {\cite{Shilow}} with significant later contributions by Zelasko, Kaplansky et al. In a Banach algebra $\mathfrak{B},$ an element $z$ is called a TDZ if there exists a sequence $\{z_n\}_{n=1}^\infty$ in $\mathfrak{B}$ with $\|z_n\|=1,$ for each $n\geq1$ such that either $zz_n \to 0$ or $z_nz\to 0$ as $n\to \infty$.

	In {\cite{Schulz}} F. Schulz, R. Brits, M. Hasse  gave a characterization of Banach algebras in which every element is a TDZ. The purpose of the present paper is to determine zero divisors and TDZ in some Banach algebras which do not fall in the preceding category.

	In this paper, we consistently employ the notation $\mathbb{N}$ to signify the set of positive integers, $\mathbb{C}$ designates the set of complex numbers, $\mathcal{B}(X)$ denotes the Banach algebra of all bounded linear operators on $X,$ $\mathcal{R}(f)$ denotes the range of a function $f,$ $[f]$ denotes the linear span of $f,$ $C(X)$ denotes the space of all complex-valued continuous functions on $X$ and $Z(f)=\{x:f(x)=0\}$ refers to the zero set of $f$.

		Let $(X,\Omega,\mu)$ be a $\sigma$-finite measure space. For $1\leq p\leq\infty$, we abbreviate the Lebesgue space $L^p(X,\Omega,\mu)$ to $L^p(\mu),$ and denote the $L^p$-norm by $\|\cdot\|_p.$ Note that $L^p(\mu)$ is a Banach space under the $L^p$-norm $(1\leq p\leq \infty)$ (see \rm{\cite{Bollobas,Sheldom,Royden}}). 
		 
		In the scenario where $X$ equals $\mathbb{N},$ and $\mu$ represents the counting measure on $\mathbb{N},$ the $L^p(\mu)$ becomes $\ell^p.$

		Let $(X,\Omega,\mu)$ be a $\sigma$-finite measure space  and $u\in L^\infty(\mu).$ Let $\phi$ be a measurable transformation on $X$ and suppose that the measure $\mu_\phi$ defined by $\mu_\phi(E)=\mu(\phi^{-1}(E))~\forall~E\in \Omega,$ is absolutely continuous w.r.t. $\mu$ (we write $\mu_\phi\ll\mu,$ as usual). Then we can define the linear transformation $uC_\phi:L^p(\mu) \to L^p(\mu)$ as $$uC_\phi f(x)= u(x)f(\phi(x))~\text{for all}~ x\in X~\text{and}~f\in L^p(\mu)$$
		and the multiplication operator $M_u:L^p(\mu)\to L^p(\mu)$ defined as $$M_u(f)(x)=u(x) f(x)~\forall~ x\in X.$$
		Assumption $\mu_\phi\ll\mu$ is used in the well-definedness of $uC_\phi$. Further, if the transformation $uC_\phi$ is bounded, then it is called a weighted composition operator on $L^p(\mu).$ Specifically, when considering the constant function $1$ as $u$, the associated operator $uC_\phi$ on $L^p(\mu)$ becomes $C_\phi$ and is called a composition operator (see \rm{\cite{Sheldom, RKM, Takagi, Zimmer}}).

		In \cite{RKOMAL}, R. K. Singh and B. S. Komal have shown that a linear transformation $C_\phi$ on $\ell^p~(1\leq p<\infty)$ defined as $$C_\phi(f)=\sum_{n=1}^{\infty}f(n)\mychi_{\phi^{-1}(n)}$$ is bounded if and only if $\{|\phi^{-1}(n)|: n\in \mathbb{N}\}$ is a bounded set, where $|\phi^{-1}(n)|$ represents the cardinality of the set $\phi^{-1}(n).$ Further, in \cite{DKHC}, it is proved that a linear transformation $uC_\phi$ on $L^p(\mu)$ is bounded if and only if  $\frac{d\mu\phi^{-1}}{d\mu}:X\setminus Z(u)\to \mathbb{C}$ is essentially bounded, where $\frac{d\mu\phi^{-1}}{d\mu}$ represents the Radon-Nikodym derivative.

		Weighted composition operators constitute a significant and diverse class of operators. There are several important reasons to study these operators. For example, Lamperti \cite{Royden} showed that every isometry on $L^p(\mu)$ spaces is a weighted composition operator. Most recently, in \cite{Strivedi}, it is shown that analytic m-isometries without the wandering subspace property are weighted composition operators which can be realised as weighted shifts, thus providing a partial answer to the problem raised by Shimorin in [\cite{Shimorin}, p. 185]. 
	
 
	The current article is aimed at finding weighted composition operators on the spaces $\ell^p$ and $L^p(\mu)$ as zero divisors in $\mathcal{B}(\ell^p)$ and  $\mathcal{B}(L^p(\mu))$ respectively. Further, we find TDZ in the Banach algebras $C(X)$ in its full generality. Consequently, we determine condition under which the multiplication operator $M_h$  is a TDZ in $\mathcal{B}(C(X)).$  We show that every element in $C(X)$ and in $L^\infty(\mu)$ is a polynomially TDZ.  We also prove that every multiplication operator $M_h$ is a polynomially TDZ in $\mathcal{B}(C(X))$ and  $\mathcal{B}(L^p(\mu))$ respectively. Finally, we show that  if $H$ is a separable Hilbert space, then each $T\in \mathcal{B}(H)$ is a strongly TDZ. Lastly, we construct a class of compact operators in $\mathcal{B}(\ell^p)$ which are also TDZ.  The paper is arranged as follows.

	In section \ref{sec2}, we state some definitions and basic results which are relevant. Section \ref{sec3.1}  is devoted for the characterization of weighted composition operators on the $\ell^p$ spaces as zero divisors in $\mathcal{B}(\ell^p)$ and in section \ref{sec3.2}, we characterize composition operators on the $L^p(\mu)$ spaces as zero divisors in $\mathcal{B}(L^p(\mu)).$ In section \ref{sec3.3}, we completely characterize the TDZ in the Banach algebras $C(X)$ in its full generality. Section \ref{sec3.4} is devoted for the characterization of polynomially TDZ. In this section we show that every element in $C(X)$ and $L^\infty(\mu)$ is a polynomially TDZ. Further, we prove that every multiplication operator $M_h$ is a polynomially TDZ in $\mathcal{B}(C(X))$ and  $\mathcal{B}(L^p(\mu))$ respectively.
	In section \ref{sec3.5}, we show that every TDZ is also a strongly TDZ. We also show that if $H$ is a separable Hilbert space, then each $T\in \mathcal{B}(H)$ is a strongly TDZ. Lastly, we construct a class of compact operators in $\mathcal{B}(\ell^p)$ which are also TDZ.
	
	\section{Preliminaries} \label{sec2}
	In this section, we state few definitions and some known results, which we shall use later.		
	\begin{definition}\rm{\cite{RG Douglas,Simmons}}\label{anurag1}
		In a Banach algebra $\mathfrak{B}$ with identity $1,$ an element $x$ is called regular if $ xy=yx=1$ for some  $y\in \mathfrak{B},$ otherwise it is called a singular element. 
	\end{definition}

	\begin{definition}\rm({\cite{Zimmer}})\label{anurag5}
		Let $(X,\Omega,\mu)$ be a measure space, $Y$ be a topological space and $f:X\to Y$ . An element $y\in Y$ is in the essential range of $f$ if $\mu(f^{-1}(U))>0$ for each neighborhood $U$ of $y.$ We denote the essential range of $f$ by ess.range$(f).$  
	\end{definition}

	\begin{theorem}\rm {\cite{RK,RKM}}\label{anurag8}
		Let $\phi:\mathbb{N} \to \mathbb{N}$ and $1\leq p<\infty.$ Let $C_{\phi}$ be the composition operator on $\ell^p.$ Then the following statements are true \begin{enumerate}
			\item [$(1)$] $C_{\phi}$ is injective $\iff$ $\phi$ is surjective;
			\item [$(2)$] $C_{\phi}$ is surjective $\iff$ $\phi$ is injective;
			\item [$(3)$] $C_{\phi}$ is invertible $\iff$ $\phi$ is invertible.
		\end{enumerate} 
	\end{theorem}
	\begin{theorem}\rm {\cite{RK,RKM}}\label{anurag17}
		Suppose $u\in \ell^\infty$ with $u(n)\neq0~\forall~n\in \mathbb{N}.$ Let $\phi:\mathbb{N} \to \mathbb{N}$ and $1\leq p<\infty.$ Then the weighted composition operator $uC_{\phi}$ on $\ell^p$ is injective if and only if $\phi$ is surjective .		 
	\end{theorem}

	\begin{definition}\rm{\cite{HCDilip}}
		Let $X$ be a non-empty set. We say that $u:X \to \mathbb{C}$ is bounded away from zero if there exists $a>0$ satisfying $0<\frac{1}{|u(x)|}\leq a~\forall~x\in X.$
	\end{definition}
	\begin{theorem}\rm {\cite{RK,RKM}}\label{anurag31}
		Suppose $u\in \ell^\infty$ is bounded away from zero. Let $\phi:\mathbb{N} \to \mathbb{N}$ and $1\leq p<\infty.$ Let $uC_{\phi}:\ell^p \to \ell^p$ be the weighted composition operator. Then $uC_{\phi}$ is surjective if and only if $\phi$ is injective.		 
	\end{theorem}

	\begin{definition}\rm{\cite{R.B.Holmes}}
		Let $H$ be an infinite dimensional Hilbert space. An operator $T \in \mathcal{B}(H)$  is called antinormal if $d(T, \mathcal N)=\inf_{N\in \mathcal N}\|T-N\|=\|T\|,$ where $\mathcal N$ is the class of all normal operators in $\mathcal{B}(H).$
	\end{definition}
	\begin{theorem}\rm{\cite{Dillip}} \label{DAntinormal}
		Suppose $u\in \ell^\infty$ is bounded away from zero and let $\phi$ be a bijective self map on $\mathbb{N}.$ Then $uC_{\phi}$ is not antinormal on $\ell^2$.
	\end{theorem}
	
	
	\begin{theorem}\rm{\cite{AnuragHarish2}} \label{ess.range(f)}
		Let	$h \in L^\infty(\mu).$ Then $h$ is a TDZ if and only if $0\in \textrm{ess.range}(h).$ 
	\end{theorem}
	\begin{theorem}\rm {\cite{AnuragHarish2}}\label{Anurag10}
		Let $h \in L^\infty(\mu)$ and $1\leq p\leq\infty.$ Then the multiplication operator $M_h$ is a TDZ in $B(L^p(\mu))$ if and only if $h$ is a TDZ in $L^\infty(\mu).$ 
	\end{theorem}
	
	\begin{theorem}\rm {\cite{Anurag}}\label{anurag20}
		Let $f\in C[a,b].$ Then $f$ is a TDZ if and only if $Z(f)\ne \emptyset.$ 
	\end{theorem}
	
	
	
	\section{Main results} \label{sec3}

\subsection{Weighted composition operators on $\ell^p~(1\leq p<\infty)$ spaces as zero divisors} \label{sec3.1}
In this section, we give a characterization of the weighted composition operators on $\ell^p,$ which are zero divisors in $\mathcal{B}(\ell^p).$\\
We begin with the following theorem, which shows that under a certain condition on $\phi$ and the zero set $Z(u)$ of the weight function $u$, the weighted composition operator $uC_\phi$ is a right zero divisor. 
\begin{theorem}\label{hc31}
	Let $u\in \ell^\infty$ and $Z(u)$ denote the zero set of $u.$ If $\phi: \mathbb{N} \to \mathbb{N}$ is such that there exists $n_0 \in \mathbb{N}$ with $\phi^{-1}(n_0)\neq \emptyset$ and $\phi^{-1}(n_0)=\{n_1, n_2, ..., n_k\} \subset Z(u)$ then $uC_\phi$ is right zero divisor.
\end{theorem}
\begin{proof}
	Let $\phi^{-1}(n_0)=\{n_1, n_2, ..., n_k\}.$ Then $T$ defined as $$T(f)=\begin{cases}
		f, &\text{ if } f=\mychi_{n_1}, \mychi_{n_2},..., \mychi_{n_k}\\
		0, &\text{ if } f=\mychi_n \text{ and } n \neq n_1, n_2,..., n_k
	\end{cases}$$
linearly extends to a bounded linear operator $T$ on $\ell^p.$\\
For any $g=\sum_{n=1}^{\infty}g_n\mychi_n \in \ell^p$
\begin{align*}
	T(g)&=\sum_{n=1}^{\infty}g_nT(\mychi_n)\\&=\sum_{i=1}^{k}g_{n_i}\mychi_{n_i}
\end{align*} 
and
 $$(T\circ uC_\phi)(\mychi_n)=\begin{cases}
	\sum_{i=1}^{k}u(n_i)\mychi_{n_i}, &\text{ if } n=n_0\\
	0, &\text{ otherwise. }\end{cases}.$$
Since $\phi^{-1}(n_0) \subset Z(u),$ therefore $T\circ uC_\phi=0.$ Hence $uC_\phi$ is right zero divisor.
\end{proof}

 The following theorem gives a necessary and sufficient condition for $uC_\phi$ to be a right zero divisor when $Z(u)=\emptyset.$
\begin{theorem}\label{Anurag31}
	Let $\phi:\mathbb{N} \to \mathbb{N}$ and suppose $u\in \ell^\infty$ is such that $u(n)\neq 0~\forall~n\geq 1.$ Then $uC_{\phi}$ is a right zero divisor if and only if  $\phi$ is not injective.\label{Anurag32}
\end{theorem}
\begin{proof}
	Suppose $\phi$ is injective and $u\in \ell^\infty$ is such that $u(n)\neq 0~\forall~n\geq 1.$ Let $T\in \mathcal{B}(\ell^p)$ be such that $T(uC_{\phi})=0.$ This implies that \begin{align*}
	T(uC_{\phi})(\mychi_n)&=T(u\mychi_{\phi^{-1}(n)})\\&=u(\phi^{-1}(n))T(\mychi_{\phi^{-1}(n)})\\&=0~~\text{for each~}n\geq 1.
	\end{align*} Since $u(n)\neq 0~\forall~n\geq 1.$ Therefore $$T(\mychi_{\phi^{-1}(n)})=0 \ \forall \ n\geq1.$$  Again, as $\phi$ is injective, hence $$\vert\phi^{-1}(n)\vert\leq 1 \ \forall \ n\in \mathbb{N} \text{ and } \bigcup_{n=1}^{\infty}\phi^{-1}(n)=\mathbb{N}.$$ Therefore, it follows that $T(\mychi_n)=0~\forall~ n\geq 1.$ Thus $T=0.$ Hence $uC_{\phi}$ can not be a right zero divisor.\\
	Conversely, suppose $\phi$ is not injective and $u(n)\neq 0~\forall~n\geq 1.$ Now we first prove that $uC_\phi$ is not surjective.\\
	Since $\phi$ is not injective, then by Theorem \ref{anurag8}, $C_\phi$ is not surjective. Hence, there exists $f_0\in \ell^p$ such that $f_0\notin \mathcal{R}(C_\phi).$ We claim that $uf_0 \notin \mathcal{R}(uC_\phi).$\\
	For, if $uf_0 \in \mathcal{R}(uC_\phi),$ then there exists $g_0\in \ell^p$ such that $uf_0=uC_\phi g_0.$ This implies that
	$$u(n)(f_0(n)-(g_0\circ \phi)(n))=0~~\text{for each }~n\geq1.$$ Since $u(n)\neq 0~\forall~n\geq 1,$ hence $f_0=g_0\circ \phi \in \mathcal{R}(C_\phi).$ This is a contradiction.
	
%
	Thus, if $\phi$ is not injective and $u(n)\neq 0~\forall~n\geq 1,$ then $uC_\phi$ is not surjective. \\
	Hence there exists an $f_0\in \ell^p \backslash \mathcal{R}(uC_\phi)$. 
	Let $M=[f_0].$ Since $M$ is a finite dimensional subspace of $\ell^p,$ it follows from Hahn-Banach theorem that there exists a closed subspace $N$ such that $\ell^p=M\bigoplus N.$ 
	 Let $T$ denote the projection on $M.$
	Then $T$ is a non-zero bounded linear operator on $\ell^p$ and $\mathcal{R}(uC_\phi)$ is contained in $N.$ Hence $$~(T\circ uC_\phi)(f)=T(uC_\phi f)=0 ~\forall ~f\in \ell^p.$$ This implies that $uC_\phi$ is a right zero divisor.
\end{proof}
We now give examples of two weighted composition operators, which are right zero divisors.
\begin{example}
	Define $\phi :\mathbb{N} \to \mathbb{N}$ as $\phi(1)=1,~ \phi(n)=n+1~\forall~n\geq2$ and $u\in \ell^\infty$ as $u(1)=0,~u(n)=\frac{1}{n}~\forall~n\geq 2.$ Then, by Theorem \ref{hc31}, $uC_\phi$ is a right zero divisor.
\end{example}
\begin{example}
	Define $\phi :\mathbb{N} \to \mathbb{N}$ as $\phi(1)=\phi(2)=1,~\phi(n)=n~\forall~n\geq 3$ and $u\in \ell^\infty$ as $u(n)=\frac{1}{n}~\forall~n\geq 1.$ Then, by Theorem \ref{Anurag31}, $uC_\phi$ is a right zero divisor.
\end{example}

\begin{theorem}\label{HCsir1}
	Let $u\in \ell^\infty$ and $\phi:\mathbb{N}\to \mathbb{N}$ be such that $uC_\phi$ is a bounded linear operator on $\ell^p.$ Then $uC_\phi$ is a left zero divisor if $\phi$ is not surjective.
\end{theorem}
\begin{proof}
	Suppose $\phi$ is not surjective. Then there exists $n_0\in \mathbb{N}$ such that $\phi^{-1}(n_0)= \emptyset.$ \\
	Let $f_0=\mychi_{n_0}$ and $M=[f_0].$ Since $M$ is a one dimensional subspace of $\ell^p,$ then there exists a closed subspace $N$ of $\ell^p$ such that $\ell^p=M\bigoplus N.$
Let $T$ denote the projection on $M.$
	Clearly, $0\neq T$ defines a bounded linear operator on $\ell^p$.
	Observe that for $f=\mychi_{n_0}$ \begin{align*}
		(uC_\phi \circ T)(\mychi_{n_0})&=u\mychi_{\phi^{-1}(n_0)}\\&=0
	\end{align*}
	and for $f\notin M$\begin{align*}
		(uC_\phi \circ T)(f)&=uC_\phi (T(f))\\&=0.
	\end{align*}	
	This implies that $uC_\phi \circ T=0.$ Hence, $uC_\phi$ is a left zero divisor.
\end{proof}
\begin{remark}
	The following example shows that the converse of the above Theorem \ref{HCsir1} is not true.
\end{remark}
\begin{example}
	Let $\phi: \mathbb{N} \to \mathbb{N}$ be defined as $\phi(2n-1)=\phi(2n)=n~\forall~n\geq1$ and $u\in \ell^\infty$ as $u(1)=u(2)=0$ and $u(n)=1~\forall~n\geq3.$ It is easy to see that null space $M$ of $uC_\phi$ is one dimensional. Let $T$ denote the projection on $M.$ Then $T$ is a non-zero bounded linear operator on $\ell^p$ and $uC_\phi\circ T=0.$ Hence $uC_\phi$ is a left zero divisor.
\end{example}
\begin{theorem}\label{anurag13}
	Suppose $u\in \ell^\infty$ is such that $u(n)\neq0~\forall~n\in \mathbb{N}$. Then $uC_{\phi}$ is a left zero divisor if and only if $\phi$ is not surjective.\label{Anurag33}
\end{theorem}

\begin{proof}
	Suppose $\phi$ is surjective. Since $u(n)\neq0~\forall~n\in \mathbb{N},$ hence by Theorem \ref{anurag17}, $uC_\phi$ is injective. Let $T\in \mathcal{B}(\ell^p)$ be such that $uC_\phi \circ T=0.$ Then $$(uC_\phi \circ T)(\mychi_n)=uC_\phi(T(\mychi_n))=0 \text{ for each } n\geq1.$$ Now injectivity of $uC_\phi$ implies that $T(\mychi_n)=0~\forall~n\geq1.$ Therefore $T=0.$ Hence, $uC_\phi$ can not be a left zero divisor. \\
	Conversely, suppose $\phi$ is not surjective. Then, again by Theorem \ref{anurag17}, $uC_\phi$ is not injective. Hence there exists $0\neq f_0\in\ell^p$ such that $uC_\phi(f_0)=0.$ Now, define a linear operator $T:\ell^p \to \ell^p$ as follows. $$T(f)=f(1)f_0~\forall ~f\in \ell^p.$$ Clearly, $T\neq0$ since $T(\mychi_1)=f_0 \neq0.$ Also, for each $f\in\ell^p,$ $$\Vert Tf\Vert=\vert f(1)\vert \Vert f_0\Vert\leq\Vert f\Vert\Vert f_0\Vert.$$ Hence $T$ is bounded.  Further, observe that for each $f\in\ell^p,$ $$uC_\phi (T(f))=uC_\phi(f(1)f_0)=f(1)uC_\phi(f_0)=0.$$ Hence $uC_\phi \circ T=0.$ Therefore $uC_\phi$ is a left zero divisor.
\end{proof}
Now we give examples of two weighted composition operators successively: one that is a left zero divisor and another that is not a left zero divisor.
\begin{example}
	Let $\phi :\mathbb{N} \to \mathbb{N}$ be defined as $\phi(n)=n^2~\forall~n\in \mathbb{N}$ and $u\in \ell^\infty$ be defined as $u(n)=\frac{1}{n}~\forall~n\geq 1.$ Then by Theorem \ref{anurag13}, $uC_\phi$ is a left zero divisor.
\end{example}
\begin{example}
	Let $\phi :\mathbb{N} \to \mathbb{N}$ be defined as $\phi(n)=n~\forall~n\geq1$ and $u\in \ell^\infty$ be defined as $u(n)=1+\frac{1}{n}~\forall~n\geq 1.$ Then by Theorem \ref{anurag13}, $uC_\phi$ is not a left zero divisor.
\end{example}
\begin{theorem}\label{Anurag34}
	Suppose $u\in \ell^\infty$ is such that $u(n)\neq0~\forall~n\in \mathbb{N}$. If $\phi$ is not invertible, then $uC_\phi \in \mathcal{B}(\ell^p)$ is a zero divisor.
\end{theorem}
\begin{proof}
	The proof follows from Theorem $\ref{Anurag31}$ and Theorem $\ref{anurag13}.$	
\end{proof}

\begin{theorem}\label{TDZ UC}
	Suppose $u\in \ell^\infty$ is bounded away from zero. Then $uC_\phi \in \mathcal{B}(\ell^p)$ is a zero divisor if and only if $\phi$ is not invertible.
\end{theorem}
\begin{proof}
	If $\phi$ is invertible then, by Theorem $\ref{anurag17},$ $uC_\phi$ is also invertible. Hence $uC_\phi$ can not be zero divisor.
	Conversely, if $\phi$ is not invertible, then by Theorem $\ref{Anurag34},$ $uC_\phi$ is a zero divisor.
\end{proof} 
The subsequent theorem establishes a connection between the antinormality of a weighted composition operator and its characteristic of being a zero divisor.
\begin{theorem}
	Let $u\in \ell^\infty$ be bounded away from zero such that $uC_\phi$ is antinormal. Then $uC_\phi$ is a zero divisor in $\mathcal{B}(\ell^2).$
\end{theorem}
\begin{proof}
	Let $u\in \ell^\infty$ be bounded away from zero. Suppose $uC_{\phi}$ is antinormal, then by Theorem $\ref{DAntinormal},$ $\phi$ is not bijective. Hence by Theorem $\ref{TDZ UC}$, $uC_{\phi}$ is a zero divisor.
\end{proof}
Here is an example of an antinormal weighted composition operator that is also a zero divisor.
\begin{example}
	Define $\phi :\mathbb{N} \to \mathbb{N}$ as $\phi(n)=n^2~\forall~n\in \mathbb{N}$ and $u\in \ell^\infty$ as $u(n)=1~\forall~n\geq 1.$ Then by Theorem \ref{anurag13}, $uC_\phi$ is a left zero divisor.
\end{example}

In the following theorem, we characterize another class of weighted composition operators, which are left zero divisors.
\begin{theorem}
	Suppose $u\in \ell^\infty$ and $\phi:\mathbb{N}\to \mathbb{N}.$ If for some $n_0\in \mathbb{N},~\phi^{-1}(n_0)\neq \emptyset$ and $u(m)=0~\forall~ m\in \phi^{-1}(n_0),$ then $uC_\phi$ is a left zero divisor.
\end{theorem}
\begin{proof}
	Let $f_0=\mychi_{n_0}$ and $M=[f_0].$ Since $M$ is a one dimensional subspace of $\ell^p,$ then there exists a closed subspace $N$ of $\ell^p$ such that $\ell^p=M\bigoplus N.$
	Let $T$ denote the projection on $M.$
	Clearly, $0\neq T$ defines a bounded linear operator on $\ell^p$.
	Observe that for $f=\mychi_{n_0}$ \begin{align*}
		(uC_\phi \circ T)(\mychi_{n_0})&=u\mychi_{\phi^{-1}(n_0)}\\&=0
	\end{align*}
	and for $f\notin M$\begin{align*}
		(uC_\phi \circ T)(f)&=uC_\phi (T(f))\\&=0.
	\end{align*}	
	This implies that $uC_\phi \circ T=0.$ Hence, $uC_\phi$ is a left zero divisor.
\end{proof}
\begin{theorem}
	Let $u\in \ell^p$ and $\phi:\mathbb{N} \to \mathbb{N}.$ If $u(n_0)=0$ for some $n_0 \in \mathbb{N},$ then $uC_\phi$ is a right zero divisor. 
\end{theorem}
\begin{proof}
	Observe that for any $f\in \ell^p,$ \begin{align*}
		(uC_\phi \circ f)(n_0)&=u(n_0)(f \circ \phi)(n_0)\\&=0 \quad(\text{since} ~u(n_0)=0).
	\end{align*}
	Let $g=\mychi_{n_0}.$ Since for each $f\in \ell^p,$  \begin{align*}
		0=(uC_\phi \circ f)(n_0)&\neq g(n_0)=1.	
	\end{align*}
	Therefore $g \notin \mathcal{R}(C_\phi).$ Hence $uC_\phi$ is not surjective. Consequently, using the arguments given in Theorem \ref{Anurag31}, $uC_\phi$ is a right zero divisor.
\end{proof}

\begin{theorem}
	Let $u\in \ell^\infty$ and $\phi:\mathbb{N} \to \mathbb{N}.$ Then $uC_\phi$ is injective if and only if $\phi$ is surjective and for each $n\geq 1,$ there exists $m_n\in \phi^{-1}(n)$ such that $u(m_n)\neq 0.$
\end{theorem}  
\begin{proof}
	Suppose $uC_\phi$ is injective. Then 
	\begin{align*}
		uC_\phi(\mychi_n)\neq 0~~\text{for each}~n\geq 1 &\implies u\mychi_{\phi^{-1}(n)}\neq 0~~\text{for each}~n\geq 1.	
	\end{align*}
	This implies that for each $n\geq 1,~\phi^{-1}(n)\neq \emptyset$ and there exists $m_n\in \phi^{-1}(n)$ such that $u(m_n)\neq 0.$\\
	Conversely, suppose for each $n\geq 1$ there exists $m_n\in \phi^{-1}(n)$ such that $u(m_n)\neq 0.$\\ 
	Then \begin{align*}
		u\mychi_{\phi^{-1}(n)}\neq 0~~\text{for each}~n\geq 1&\implies uC_\phi(\mychi_n)\neq 0~~\text{for each}~n\geq 1.	
	\end{align*}
	Therefore $uC_\phi$ is injective.
\end{proof}
The subsequent corollary is a straightforward outcome of the aforementioned theorem.
\begin{corollary}
	Let $u\in \ell^\infty$ and $\phi:\mathbb{N} \to \mathbb{N}.$ Then $uC_\phi$ is not a left zero divisor if and only if $\phi$ is surjective and for each $n\geq 1$ there exists $m_n\in \phi^{-1}(n)$ such that $u(m_n)\neq 0.$ 
\end{corollary}
\subsection{ Composition operators on $L^p(\mu)~(1\leq p\leq \infty)$ spaces as zero divisors }\label{sec3.2}\

In this section, we study composition operators on $L^p(\mu)$ spaces, which are zero divisors.
\begin{theorem}
	Suppose $\phi:X \to X$ is a non-singular measurable transformation such that $\frac{d\mu\phi^{-1}}{d\mu}$ is bounded away from zero on the complement of  $Z(\frac{d\mu\phi^{-1}}{d\mu}).$ Then $C_\phi\in \mathcal{B}(L^2(\mu))$ is surjective if and only if $\phi$ is injective almost everywhere.
\end{theorem}
\begin{proof}
	Let $\phi:X \to X$ be a non-singular measurable transformation such that $\frac{d\mu\phi^{-1}}{d\mu}$ is bounded away from zero on the complement of zero-set of  $\frac{d\mu\phi^{-1}}{d\mu}.$ Then by Theorem 2.2 \cite{RKAshok} range of $C_\phi$ is closed in $L^2(\mu)$. \\
	Suppose $\phi$ is injective almost everywhere. Let $E$ be any measurable set with $0<\mu(E)<\infty.$ Then there exists a measurable set $E_0$ with $\mu(E_0)=0$ and $\mu(E)=\mu(E\cap(X-E_0)).$ \\Let $F=E\cap(X-E_0)$ and $f=\mychi_{\phi(F)}.$ Then $C_\phi(f)=\mychi_F=\mychi_E.$ Thus for each measurable set $E$, with $0<\mu(E)<\infty$, there exists $f\in L^2(\mu)$ such that $C_\phi(f)=\mychi_E.$ This proves that $\mathcal{R}(C_\phi)$ is dense in $L^2(\mu)$. Consequently, $C_\phi$ is surjective.\\
	Conversely, suppose $C_\phi$ is surjective. To show $\phi$ is injective almost everywhere. \\If not, then there exist two distinct measurable sets $E_1$ and $E_2$ with $0<\mu(E_1)<\infty,~0<\mu(E_2)<\infty$ such that $\mu(E_1-E_2)>0,~\mu(E_2-E_1)>0$ and $\phi(E_1)=\phi(E_2).$ Let $g=\mychi_{E_1}\in L^2(\mu).$ Since $C_\phi$ is surjective, hence there exists $f\in L^2(\mu)$ such that $$C_\phi(f)=\mychi_{E_1} \implies f\circ\phi=\mychi_{E_1}.$$ Since $\phi(E_1)=\phi(E_2),$ therefore \begin{align*}
		1&=f\circ \phi(E_1)\\&=f\circ \phi(E_2)\\&=\mychi_{E_1}(E_2)\\&=0 \quad on ~E_2-E_1
	\end{align*}
	which is a contradiction because $\mu(E_2-E_1)>0$. Therefore $\phi$ must be injective almost everywhere.
\end{proof}
\begin{theorem}
	$C_\phi\in \mathcal{B}(L^p(\mu))$	is a left zero divisor if and only if there exists a measurable set $E_0$ with $0<\mu(E_0)<\infty$ such that $\mu(\phi^{-1}(E_0))=0.$
\end{theorem}

\begin{proof}
	Suppose there exists a measurable set $E_0$ with $0<\mu(E_0)<\infty$ satisfying $\mu(\phi^{-1}(E_0))=0.$ Then for $u=1,$ Theorem $\ref{amar1}$ implies that $C_\phi$ is a left zero divisor.\\
	Conversely, suppose there does not exist any measurable set $E_0$ with $0<\mu(E_0)<\infty$ satisfying $\mu(\phi^{-1}(E_0))=0.$ Then for each measurable set $E$ with $\mu(\phi^{-1}(E))=0$ implies $\mu(E)=0.$ Hence by Theorem \ref{amar2}, $C_\phi$ is injective. Let $C_\phi \circ T=0.$ Then for each measurable set $E,$ injectivity of $C_\phi$ implies that $T(\mychi_E)=0.$ This implies that $T=0.$ Hence $C_\phi$ can not be left zero divisor.\\This completes the proof.
\end{proof}
\begin{theorem}\label{amar1}
	Let $uC_\phi \in \mathcal{B}(L^p(\mu)).$ Then it is a left zero divisor if there exists a measurable set $E_0$ with $0<\mu(E_0)<\infty$ such that $\phi^{-1}(E_0)=A\cup B$ with $u=0$ a.e. on $A$ and $\mu(B)=0.$ 	
\end{theorem}
\begin{proof}
	Suppose there exists a measurable set $E_0$ with $0<\mu(E_0)<\infty$ such that $\phi^{-1}{(E_0)}=A\cup B,$ where $u=0$ a.e. on $A$ and $\mu(B)=0.$ 
Let $f_0=\mychi_{E_0}$ and $M=[f_0].$
Let $T$ denote the projection operator on $M.$
Clearly, $0\neq T$ is a bounded linear operator on $L^p(\mu)$. Now for the measurable set $E_0,$
	\begin{align*}
		(uC_\phi\circ T)(\mychi_{E_0})&=uC_\phi (T(\mychi_{E_0}))\\
		&=uC_\phi (\mychi_{E_0})\\
		&=u\mychi_{\phi^{-1}(E_0)}\\
		&=0 \quad(\text{since}~ \phi^{-1}(E_0)=A\cup B ~\text{and}~ u=0 ~\text{a.e. on}~ A~ \text{and}~ \mu(B)=0)
	\end{align*}
and for $f\notin M$\begin{align*}
	(uC_\phi \circ T)(f)&=uC_\phi (T(f))\\&=0.
\end{align*}
	Hence $uC_\phi\circ T=0.$ Therefore $uC_\phi$ is a left zero divisor in $\mathcal{B}(L^p(\mu)).$	
\end{proof}

\subsection{Topological divisors of zero in $C(X)$}\label{sec3.3}\

Consider $X$ as a compact Hausdorff space, and let $C(X)$ represent the Banach algebra consisting of all continuous complex-valued functions on X. In \cite{Azarpanah}, F. Azarpanah, D. Esmaeilvandi and A.R. Salehi characterized the regular elements in $C(X).$
With this context, the following results have been obtained:

\begin{theorem}\label{CHS TDZ}
	Let $f\in C(X).$ Then $f$ is a TDZ if and only if it is singular in $C(X).$    
\end{theorem}
\begin{proof}
	If $f$ is a TDZ in $C(X)$ then it is clearly a singular element.\\ Conversely, suppose that $f\in C(X)$ is a singular. Then there exists $x_0\in X$ such the $f(x_0)=0.$ Now, by continuity of $f,$ for each $n\geq 1,$ there exists an open set $E_n$ containing $x_0$ such that $$|f(x)|<\frac{1}{n}~ \text{for each}~ x\in E_n.$$ Let $A_n=E^c_n$ and $B_n=\{x_0\}.$ By Urysohn’s Lemma, there exists a function $f_n:X \to [0,1]$ such that $$f_n(B_n)=1 ~\text{and}~ f_n(A_n)=0.$$ Clearly $\|f_n\|=1$ and $\|ff_n\|<\frac{1}{n}.$ This implies that $\|ff_n\| \to 0$ as $n \to \infty.$ Hence $f$ is a TDZ.
\end{proof}
\begin{theorem}\label{M_h Compact}
	Let $h\in C(X)$ and $M_h:C(X) \to C(X)$ be the multiplication operator defined as $M_h f=h f.$ Then $M_h$ is a TDZ in $\mathcal{B}(C(X))$ if and only $h$ is a TDZ in $C(X)$. 
\end{theorem}
\begin{proof}
	Suppose $h \in C(X)$ be a TDZ. Let $\{h_n\}_{n=1}^{\infty}$ be a sequence in $C(X)$ with $\Vert h_n\Vert=1 \ \forall \  n\in \mathbb{N}$ such that $\Vert h h_n\Vert \to 0$ as $n\to \infty.$\\ For each $n\geq1,$ define $M_{h_n}:C(X) \to C(X)~ \text{by}~ M_{h_n}(f)=h_n f.$ Observe that $$\Vert M_{h_n}\Vert=\Vert h_n\Vert=1 \ \forall \ n\in \mathbb{N} ~\text{and}~ M_hM_{h_n}(f)=h h_n f.$$ Then $\Vert M_hM_{h_n}\Vert=\Vert h h_n\Vert \to 0 ~\text{as}~ n\to \infty.$ Therefore, $M_h$ is a TDZ in $\mathcal{B}(C(X))$. \\
	Conversely, suppose $h \in C(X)$ is not a TDZ in $C(X)$. Then $h$ is a regular element in $C(X)$. Therefore there exists a $g \in C(X)$ such that $h \cdot g=g \cdot h=1.$  This implies that $M_{h g}=\text{I}.$ Hence $M_{h g}=M_{h}M_{g}=M_{g}M_{h}=\text{I}.$ This shows that $M_h$ is regular element in $\mathcal{B}(C(X))$. Hence $M_h$ can not be a TDZ.\\
\end{proof}
\begin{theorem}
	Let $\phi:X\to X$ be a continuous map and $C_\phi:C(X) \to C(X)$ be the composition operator defined as $C_\phi f=f\circ \phi.$ Then $C_\phi$ is a bounded linear operator on $C(X)$ with $\|C_\phi\|=1.$\\
	Further, the following statements hold.\\
	(1) If $\phi$ is bijective, then $C_\phi$ is invertible in $\mathfrak{B}(C(X)).$\\
	(2) $C_\phi$ is injective if and only if $\phi$ is surjective.\\
	(3) $C_\phi$ is surjective if and only if $\phi$ is injective.
\end{theorem}
\begin{proof}
	(1) If $\phi:X\to X$ is bijective and continuous, then it is a homeomorphism. Hence $C^{-1}_\phi=C_{\phi^{-1}}.$\\
	(2) Suppose $\phi$ is surjective and let $C_\phi f=0.$ Then $(f\circ\phi)(x)=0~\forall~x\in X.$ Now, surjectivity of $\phi$ implies that $f(x)=0~\forall~x\in X.$ Therefore $f=0.$ This implies $C_\phi$ is injective.\\
	Conversely, suppose $\phi$ is not surjective. Then $\phi(X)$ is a proper closed and compact subset of $X.$ Since $X$ is completely regular, therefore for $x_0\in X\setminus \phi(X)$ there exists a non-zero continuous function $f_0:X\to [0,1]$ such that $$f_0(x_0)=1 ~\text{and}~ f_0(\phi(X))=0.$$
	This implies that $C_\phi f_0=0.$ Since $f_0\neq 0,$ therefore $C_\phi$ is not injective. This completes the proof.\\
	(3) Suppose $\phi$ is injective. Let $g\in C(X)$ and $\phi^{-1}:\phi(X)\to X.$ Define $f=g\circ \phi^{-1}.$ By Tietz extension theorem, $f$ can be continuously extended on $X.$ Clearly, \begin{align*}
		C_\phi(f)&=f\circ \phi \\&=g\circ \phi^{-1}\circ \phi\\&=g.
	\end{align*}
	Hence $C_\phi$ is surjective.\\
	Conversely, suppose $\phi$ is not injective. Then there exists two distinct elements $x_1, x_2\in X$ such that $\phi(x_1)=\phi(x_2)$ and a continuous function $f_0:X\to [0, 1]$ such that $$f_0(x_1)=0,~f_0(x_2)=1.$$
	We, claim that $f_0\notin \mathcal{R}(C_\phi).$ Observe that if $g\in \mathcal{R}(C_\phi),$ then $g=f\circ \phi$ for some $f\in C(X).$\\
	Then, \begin{align*}
		g(x_1)=f\circ \phi(x_1)=f\circ \phi(x_2)=g(x_2).
	\end{align*}
	This implies that $f_0\notin \mathcal{R}(C_\phi).$ Hence $C_\phi$ is not surjective.
\end{proof}
\subsection{ Polynomially topological divisors of zero }\label{sec3.4}\

Motivated by the notion of polynomially compact operators\cite{Olsen}, we give the the following definition.
\begin{definition}
	Let $\mathfrak{B}$ be a Banach algebra. An element $z\in \mathfrak{B}$ is said to be polynomially TDZ if $p(z)$ is a TDZ for some non-zero complex polynomial $p$.
\end{definition}
\begin{theorem}
	Every element in $C(X)$ is a polynomially TDZ. 
\end{theorem}
\begin{proof}
	If $f\in C(X)$ is a TDZ. For the polynomial $p(x)=x,~f$ becomes a polynomially TDZ. \\
	Now, let $f\in C(X)$ be such that $f(x) \neq0 ~\forall~x\in X.$ Let $x_0\in X$ and $p(x)=x-f(x_0).$ Then, $p(f)(x_0)=0.$ Then $p(f)$ is a TDZ in $C(X).$
\end{proof}
\begin{theorem}
	Let $X$ be a compact Hausdorff space and $h:X \to X$ be a continuous map. Let $M_h:C(X) \to C(X)$ be the multiplication operator defined as $M_h f=h f.$ Then $M_h$ is a polynomially TDZ in $\mathcal{B}(C(X))$.   
\end{theorem}
\begin{proof}
	Let $x_0 \in X$ and $p(z)=z-h(x_0).$ Clearly $p(h(x))\in C(X)$ and $p(h(x_0))=0.$ Hence by Theorem $\ref{CHS TDZ},~ p(h(x))$ is a TDZ in $C(X)$ which is also a singular element. For each $f\in C(X)$ $$p(M_h)f=(M_h-h(x_0)I)f$$$$=(h-h(x_0))f$$$$=M_{h-h(x_0)}f.$$ Hence $p(M_h)=M_{p(h)}.$ By Theorem \ref{M_h Compact}, $p(M_h)$ is a TDZ in $\mathcal{B}(C(X))$. Hence $M_h$ is a polynomially TDZ in $\mathcal{B}(C(X))$.
\end{proof}

\begin{theorem}
	Let $\mathfrak{B}$ be a Banach algebra and let $z\in \mathfrak{B}$ be the TDZ. Then $z$ is also a polynomially TDZ.
\end{theorem}
\begin{proof}
	Let $z\in \mathfrak{B}$ be the TDZ in $\mathfrak{B}.$ For the polynomial $p(x)=x,$ clearly $p(z)$ is a TDZ. Hence $z$ is a polynomially TDZ in $\mathfrak{B}.$
\end{proof}
\begin{remark}
	The converse of above theorem is not true. 	Let $\mathfrak{B}$ be an arbitrary unital Banach algebra with the unit element denoted as $e$ in $\mathfrak{B}.$ Let $p(x)=x-1,$ then $p(e)=e-e=0$ which is a TDZ, while identity $e$ is not a TDZ in $\mathfrak{B}$.
\end{remark}

\begin{theorem}\label{PTDZ L^inf}
	In $L^\infty(\mu)$ every element is a polynomially TDZ.
\end{theorem}
\begin{proof}
	Let $h\in L^\infty(\mu)$ and $\alpha \in \textrm{ess.range}(h).$ Then for the polynomial $p(x)=x-\alpha,$ clearly $0\in \textrm{ess.range}(p(h)).$ Hence by the Theorem \ref{ess.range(f)}, $p(h)$ is a TDZ. Therefore $h$ is a polynomially TDZ. 
\end{proof}
\begin{theorem}
	Let $h \in L^\infty(\mu)$ and $1\leq p\leq\infty.$ Define the multiplication operator $M_h:L^p(\mu) \to L^p(\mu)$ by $M_h(f)=h \cdot f=h f$. Then $M_h$ is a polynomially TDZ in $\mathcal{B}(L^p(\mu)).$
\end{theorem}
\begin{proof}
	Let $h \in L^\infty(\mu).$ Then by Theorem \ref{PTDZ L^inf} there exists a non-zero polynomial $p$ such that $p(h)$ is a TDZ. Consequently by the Theorem \ref{Anurag10}, $p(M_h)=M_{p(h)}$ is a TDZ in $\mathcal{B}(L^p(\mu)).$ Hence $M_h$ is a polynomially TDZ in $\mathcal{B}(L^p(\mu)).$  	
\end{proof}

\subsection{Strongly topological divisors of zero}\label{sec3.5}\

The following notion is a generalization of the concept of TDZ.
\begin{definition}
	Let $X$ be a non-empty set. An operator $T\in \mathcal{B}(X)$ is called strongly TDZ if there exists a sequence $\{T_n\}_{n=1}^\infty$ in $\mathcal{B}(X)$ such that 
	\begin{enumerate}
		\item $\|T_n\|=1~\forall~n\in \mathbb{N}$
		\item and $\|T_nTx\|\to 0$ or $\|TT_nx\|\to 0$ as $n \to 0~~\forall x\in X.$  
	\end{enumerate}
\end{definition}

\begin{theorem}\label{STD1}
	Let $X$ be a Banach space. If $T \in \mathcal{B}(X)$ is a TDZ then $T$ is a strongly TDZ.
\end{theorem}
\begin{proof}
	If $T \in \mathcal{B}(X)$ is a TDZ then there exists a sequence $\{T_n\}_{n=1}^\infty$ such that $\|T_n\|=1~\forall~ n\in \mathbb{N}$ and $\|T_nT\|\to 0$ or $\|TT_n\|\to 0$ as $n \to \infty.$ Let $x\in X$ then $$\|TT_nx\|\leq \|TT_n\|\|x\|\to 0 ~or~ \|T_nTx\|\leq \|T_nT\|\|x\|\to 0 ~as ~n \to 0.$$ Thus $TT_nx \to 0$ or $T_nTx \to 0~\forall x\in X.$ Hence $T$ is a strongly TDZ.
\end{proof}
\begin{remark}
	The subsequent example illustrates that the converse of the above theorem \ref{STD1} is not true.
\end{remark}
\begin{example}
	Let $T=I\in \mathcal{B}(\ell^p)~(1\leq p\leq \infty)$ be the indentity operator. For each $n\in \mathbb{N},$ define $T_n:\ell^p\to \ell^p$ as $$T_n(x)=(0,0,...,x_{n+1},0,0,...).$$ Clearly $\|T_n\|=1~\forall~n\in \mathbb{N}$ and $\|TT_n(x)\|=|x_{n+1}|\to 0$ as $n \to \infty.$ Hence $T$ is a strongly TDZ but not a TDZ in $\mathcal{B}(\ell^p)$.  
\end{example}
\begin{theorem}\label{STD2}
	Let $H$ be a separable Hilbert space and $T\in \mathcal{B}(H).$ Then $T$ is a strongly TDZ in $\mathcal{B}(H)$.
\end{theorem}
\begin{proof}
	Let $\{e_n:n\geq1, \|e_n\|=1\}$ be an orthonormal basis for $H$. Then for each $x\in H$ there exists a unique sequence of scalars $\{\alpha_n\}_{n=1}^\infty$ such that $x=\sum_{k=1}^{\infty}\alpha_ke_k.$ For each $n\in \mathbb{N},$ define $T_n: H \to H$ by $$T_n(x)=\sum_{k=n+1}^{\infty}\alpha_ke_k.$$ Clearly $\|T_n(x)\|\leq \|x\|$ and $T_n(e_{n+1})=e_{n+1}.$ This implies that $\|T_n\|=1~\forall~n\in \mathbb{N}.$ Now observe that $$\|T_n(x)\|=\|\sum_{k=n+1}^{\infty}\alpha_ke_k\| \to 0 ~\text{as} ~n \to \infty.$$ Hence for each $T\in \mathcal{B}(H)$ and $H\in H$ $$\|TT_n(x)\| \leq \|T\|\|T_n(x)\|\to 0 ~as ~n \to \infty.$$ Therefore $T$ is a strongly TDZ in $\mathcal{B}(H)$.
\end{proof} 
\begin{remark}
	We conjecture that in the above theorem $H$ can be replaced by any Banach space with Schauder basis.
\end{remark}
\begin{remark}
	Since the space $\ell^p~(1\leq p\leq \infty)$ has $\{e_n\}_{n=1}^\infty$ as a Schauder basis, where each $e_n$ is defined as $e_n=(0,0,...,1,0,0,...).$ Therefore by Theorem \ref{STD2}, each $T\in \mathcal{B}(\ell^p)~(1\leq p<\infty)$ is a strongly TDZ in $\mathcal{B}(\ell^p)$.
\end{remark}
In the following example, we construct a class of TDZ in $\mathcal{B}(\ell^p)~(1\leq p\leq \infty).$ 

\begin{example}\label{comp}
	Let $c_0=\{(y_n)_{n=1}^\infty: y_n\in \mathbb{C}~\forall~ n\in \mathbb{N}~\text{and}~y_n \to 0~\text{as}~n\to \infty\}.$ Let $y\in c_0$ and $T:\ell^p \to \ell^p$ be defined by $T(x):=(x_1y_1,x_2y_2,...,x_ny_n,...)$. Then $T$ is TDZ in $\mathcal{B}(\ell^p).$
\end{example}
\begin{proof}
	For each $n\in \mathbb{N},$ define $T_n:\ell^p \to \ell^p$ by $$T_n(x)=(0,0,...,x_{n+1}y_{n+1},x_{n+2}y_{n+2},...).$$ Then $\|T_n\|=1$ and $T_nT(x)=(0,0,...,x_{n+1}y_{n+1},x_{n+2}y_{n+2},...).$ Clearly, for each $x\in \ell^p$ $$\|T_nTx\|\leq \sup_{k\geq n+1}|y_k|\|x\|.$$ Since $y\in c_0,$ then $\sup_{k\geq n+1}|y_k|\to 0$ as $n \to 0.$ Therefore $\|T_nTx\|\to 0$ as $n \to 0.$ Hence $T$ is a TDZ.
\end{proof}
\begin{remark}
	We note that the class of operators in the above Example \ref{comp} forms a class of compact operators.
\end{remark}





\begin{thebibliography}{20}	
	
	\bibitem{Bollobas}
	S.Abbott, Linear analysis: an introductory course by Béla Bollobás, Pp. 240.£ 16.95. 1999. ISBN 0 521 65577 3 (Cambridge University Press), The Mathematical Gazette 83, no. 497, (1999), 365-366.
	
	
	\bibitem{Strivedi}
	A.Anand, S.Chavan, S.Trivedi, Analytic m-isometries without the wandering subspace property. Proceedings of the American Mathematical Society, 148(5), (2020), pp.2129-2142.
	
	\bibitem{Sheldom}
	S.Axler, Measure, integration $\&$ real analysis. Springer Nature; (2020).
	
	\bibitem{Azarpanah}
	F.Azarpanah, D.Esmaeilvandi, A.R.Salehi. "Depth of ideals of C (X)." Journal of Algebra 528, (2019), 474-496.
	
	
	
	
	\bibitem{Anurag}
	H.Chandra, A.K.Patel, A characterization of zero divisors and topological divisors of zero in C$[a, b]$ and $\ell^\infty$. Communications of the Korean Mathematical Society 38(2), (2023), 451–459
	
	
	
	\bibitem{RG Douglas}
	R.G.Douglas, Banach algebra techniques in operator theory. Springer Science and Business Media, (2012), Vol. \textbf{179} .
	
	
	
	
	\bibitem{Gelfand Silov}
	I.M.Gel'fand, D.A.Raikov, G.E.Shilov, Commutative normed rings, New York: Chelsea Publishing Company, (1964), P. \textbf{(306)} .
	
	
	\bibitem{R.B.Holmes}
	R.B.Holmes, Best approximation by normal operators. J. Approx. Theory 12, (1974), 412–417.
	
	\bibitem{Dillip}
	D.Kumar , H.Chandra, Antinormal weighted composition operators. In Abstract and Applied Analysis (2016), Jan 1 (Vol. 2016). Hindawi.
	
	\bibitem{HCDilip}
	D.Kumar, A Study on Antinormal Composition Operators, Thesis, Banars Hindu University, (2016).
	\bibitem{DKHC}
	D.Kumar, H.Chandra, Antinormal weighted composition operators on $L^2(\mu)$-space of an atomic measure space. Arabian Journal of Mathematics. (2020), Apr;9(1):137-43.
	
	
	\bibitem{Olsen}
	C.L.Olsen, A Structure Theorem for Polynomially Compact Operators. American Journal of Mathematics, 93(3), (1971), 686–698.
	
	\bibitem{AnuragHarish2}
	A.K.Patel, H.Chandra, A Characterization of zero divisors and topological divisors of zero in some Banach algebras (Communicated paper).
	
	
	
	\bibitem{Royden}
	H.L.Royden, Real analysis (Third edition), New York, Macmillan (1988).
	
	%
	%
	
		\bibitem{Schulz}
	F.Schulz, R.Brits, M. Hasse, ``Identities, approximate identities and TDZ in Banach algebras." Journal of Mathematical Analysis and Applications 455, no. 2 (2017), 1627-1635.
	
	\bibitem{Shilow}
	G.Shilov, On regular normed rings, 
	Trudy Matematicheskogo Instituta imeni VA Steklova \textbf{21}: (1947), 3-118 .
	
	\bibitem{Shimorin}
	S.Shimorin, Wold-type decompositions and wandering subspaces for operators close to isometries, J. Reine Angew. Math. 531, 147–189
	to isometries, J. Reine Angew. Math. 531 (2001), 147–189,
	
	\bibitem{Simmons}
	G.F.Simmons, Introduction to Topology and Modern Analysis, McGraw Hill, New York, (1963).
	
	\bibitem{RK}
	R.K.Singh, B.S.Komal, Composition operator on $V$ and its adjoint, 
	Proc. Amer. Math. Soc.\textbf{70}, (1978), 21-25 .
	
	\bibitem{RKAshok}
	R.K.Singh, A.Kumar, "Multiplication operators and composition operators with closed ranges." Bulletin of the Australian Mathematical Society 16, no. 2 (1977): 247-252.
	
	\bibitem{RKM}
	R.K.Singh, J.S.Manhas, Composition operators on function space, North Holland Mathematics studies 179, Elsevier sciences publishers Amsterdam, New York, (1993).
	
	\bibitem{RKOMAL}
	R.K.Singh, B.S.Komal, ``Composition operator on $\ell^p$ and its adjoint." Proceedings of the American Mathematical Society 70, no. 1 (1978), 21-25.
	
	\bibitem{Takagi}
	Takagi, Hiroyuki. “Compact Weighted Composition Operators on Lp.” Proceedings of the American Mathematical Society 116, no. 2 (1992), 505–11.	
	
	
	\bibitem{Zimmer}
	R.J.Zimmer, Essential Results of Functional Analysis, University of Chicago Press, (1990).
	
	
	
\end{thebibliography}
\end{document}